\documentclass[12pt]{amsart}
\usepackage{color,graphicx,array, amssymb, amscd,slashed}
\usepackage{enumitem}
\newlist{steps}{enumerate}{1}
\setlist[steps, 1]{label = Step \arabic*:}

\newtheorem{themain}{Theorem}

\newtheorem{Theorem}{Theorem}[section]
\newtheorem{Lemma}[Theorem]{Lemma}

\theoremstyle{definition}

\theoremstyle{remark}

\newtheorem{Remark}[Theorem]{Remark} 

\numberwithin{equation}{section}
\newcommand{\R}{\mathbb R}
\newcommand{\id}{\operatorname{id}}
\newcommand{\SL}{{\rm SL} (2n+1, \mathbb R)}
\renewcommand{\sl}{\mathfrak{sl} (2n+1, \mathbb R)}
\newcommand{\ad}{\operatorname{Ad}}
\newcommand{\tr}{\operatorname{Tr}}
\newcommand{\sym}[2]{\operatorname{Sym}^{#2}({#1}, \mathbb R)}
\newcommand{\syms}[2]{\operatorname{Sym}_*^{#2}({#1}, \mathbb R)}
\newcommand{\skw}[1]{\operatorname{Skew}({#1}, \mathbb R)}
\newcommand{\SO}[2]{\operatorname{SO}({#1} {#2})}
\newcommand{\OO}[2]{\operatorname{O}({#1} {#2})}
\newcommand{\so}[2]{\operatorname{\mathfrak{so}}({#1\,} {#2})}

\usepackage{amsmath}	
\begin{document}
\title{Geodesics of multivariate normal distributions and 
a Toda lattice type Lax pair}
\dedicatory{}
 \author[S.-P.~Kobayashi]{Shimpei Kobayashi}
 \address{Department of Mathematics, Hokkaido University, 
 Sapporo, 060-0810, Japan}
 \email{shimpei@math.sci.hokudai.ac.jp}
 \thanks{The author is partially supported by Kakenhi 22K03304.}
 \subjclass[2020]{Primary~53C35, Seondary~62B10, 17B80}
 \keywords{Multivariate normal distributions; geodesics; symmetric spaces; Riemannian submersions; Toda lattice}
 \date{\today}
\pagestyle{plain}
\begin{abstract}
 We study geodesics of multivariate normal distributions 
 with respect to the Fisher metric.
 First it will be shown that 
 a computational formula for geodesics 
 can be understood using the block 
 Cholesky decomposition and a natural Riemannian submersion. 
 Next a mid point algorithm for geodesics will be obtained. 
 And finally a new Toda lattice type Lax pair will be derived 
 from the geodesic and the block Cholesky decomposition.
  \end{abstract}
\maketitle    

\section{Introduction and the main results}
Due to the invariance property under sufficient statistics \cite{Che}, it is important to 
consider a parametric statistical model, that is a family of 
probability distributions with 
a finite number of parameters, 
as a Riemannian manifold with specific metric, the so-called 
\textit{Fisher metric}:
\[
 g(X,Y)=E_{\theta }[(X\log p(x, \theta)) (Y \log p(x, \theta)) ],\quad 
E_{\theta}[f]
 =\int_{\R^n} f(x) p(x, \theta) \, dx,
 \]
 where $p(x, \theta)$ denotes a family of probability density 
 functions of $x \in \R^n$, and 
 $X, Y$  are tangent vectors at a point $\theta$
 of the manifold.
Fundamental parametric models are
given by the model of normal distributions and its multivariate generalization, the \textit{multivariate normal distributions model}, that is,
the probability density function is given by  
\[
 p\left(x, (\Sigma, \mu) \right)= \frac1{\sqrt {(2\pi )^{n}\det \Sigma}}
 \, {\exp \left(-\frac12(x-\mu)^{T}{\Sigma}^{-1}(x-\mu )\right)},
\]
 where $x\in \R^n$, and  $\mu \in \R^n$ is called the \textit{mean vector}
 and $\Sigma$ is called the \textit{covariance matrix}, which is 
 in order $n$  positive definite matrices, $\sym{n}{+}$.  In this paper, the $\frac12 n(n+3)$-dimensional manifold
 $\mathcal N = \left\{(\Sigma, \mu)\mid \Sigma \in \sym{n}{+}, \,\mu \in \R^n\right\}$ with the Fisher metric $g$ 
 will be called the \textit{multivariate 
 normal manifold} $(\mathcal N, g)$. See \cite{Sko} for Riemannian geometry of $(\mathcal N, g)$, and note that it has been shown \cite{KO, FIK, GQ} that $(\mathcal N, g)$ has a solvable Lie group structure. 

 It is a fundamental problem to compute geodesics of a Riemannian manifold.
 For the multivariate normal manifold $(\mathcal N, g)$, it was first given by Eriksen \cite{Eri} as follows:
 First we realize the multivariate normal manifold $\mathcal N$ as
 \begin{align}\label{eq:N}
 \mathcal N &= \left\{
 \begin{pmatrix}
\Theta &  \delta \\
 \delta^T & 1 + \|\delta \|^2_{\Theta^{-1}}
\end{pmatrix}\,\mid\,
\begin{array}{c}
\Theta = \Sigma^{-1} \in \sym{n}{+},  \\
\delta = \Sigma^{-1} \mu \in \R^n,\, \|\delta \|^2_{\Theta^{-1}}= \delta^T \Theta^{-1} \delta
\end{array}
 \right\} \\
 \nonumber
 &\subset\sym{n+1}{+},
 \end{align}
 and the geodesic equation (integrated once) can be formulated as
\[
\frac{d}{dt} \begin{pmatrix}
\Theta & \delta
\end{pmatrix} = \begin{pmatrix}
-A_0 &a_0    
\end{pmatrix} \begin{pmatrix} \Theta & \delta \\ \delta^T & 
1 + \| \delta\|^2_{\Theta^{-1}}\end{pmatrix},
\]
 where $A_0 \in \sym{n}{}$ and $a_0 \in \R^n$ are some constant.
 Second, to construct a solution of the geodesic equation, let 
 \begin{align}\label{eq:GtiniV}
 G(t) &= \exp ( t V) \quad \mbox{with} \quad
 V = \begin{pmatrix}
 - A_0 & a_0 & 0 \\
 a_0^T & 0 & - a_0^T \\
 0 & -a_0 & A_0 \\
 \end{pmatrix},
 \quad \mbox{and} \quad
  J=
\begin{pmatrix}
&&\id_{n}\\
& \id_{1}& \\
\id_{n}&&
\end{pmatrix},
\end{align}
 that is, $V \in \sym{2n+1}{}$ with $\operatorname{tr} V =0$.
 Since $G$ is positive definite with determinant $1$ and satisfies a symmetry relation $J G^{-1} J = G$, the block Cholesky decomposition  of $G= M D M^T$, see Lemma \ref{lem:geo}, is given by 
 \begin{equation}\label{eq:MD}
 M = \begin{pmatrix}
  \id_n & 0  &0 \\ \delta^T \Theta^{-1}&  1 & 0 \\
   * &  -  \Theta^{-1}\delta & \id_n \\
 \end{pmatrix}, \quad 
  D=
 \begin{pmatrix}
 \Theta & 0  & 0 \\
 0 & 1 & 0  \\
 0 & 0 & \Theta^{-1} 
 \end{pmatrix},
 \end{equation}
 where $\id_n$ is the identity matrix of order $n$, 
 $\delta : \R \to \R^n$ and $\Theta : \R \to \sym{n}{+}$ and $*$ denote 
 some $n \times n$ matrix.
 Finally from the form of $G(t)$ given by \eqref{eq:MD}, the sub-matrix of $G(t)$ by the first $(n+1)$-rows and the first 
 $(n+1)$-columns defines a curve in $(\mathcal N, g)$ and it is in fact a geodesic in $\mathcal N$.
  Note that in \cite{Eri} the formula was derived by a direct computation without using the block 
  Cholesky decomposition, and a more explicit 
 formula of the geodesic was derived in \cite{CO}.
 In Section \ref{subsc:explictgeod}, we will review the construction
 using the block Cholesky decomposition in details.
  
 We now explain the above result more geometrically: 
 The idea is that the curve $G(t)$ in \eqref{eq:GtiniV} seems to define a geodesic in 
 some homogeneous space, and the projection to the sub-matrix 
 could be understood geometrically. First the curve $G(t)$ is in $\mathcal L = \syms{2n+1}{+}$, where $\syms{2n+1}{+}$ denotes the set of positive definite matrices with determinant $1$ and 
 $\mathcal L$ is a Riemannian symmetric space of 
 type \textrm{AI}:
 \[
 \mathcal L =  \SL/\SO{2n+1}{},
 \]
  and thus $G(t)$ is a homogeneous geodesic in $\mathcal L$. Moreover, there is a non-compact Riemannian symmetric space of type \textrm{BDI}:
 \begin{equation}\label{eq:RiemSym}
\mathcal M= \SO{n+1, }{n}/\operatorname{S}(\OO{n+1}{} \times \OO{n}{})
 \subset \mathcal L= \syms{2n+1}{+},
\end{equation}
 which is a totally geodesic submanifold of $\mathcal L$. 
 Note that $\mathcal M$ is in fact a reflective submanifold (a special totally geodesic submanifold) of $\mathcal L$,  see \cite{Le}.
 It is easy to see that a point $m \in \mathcal L = \syms{2n+1}{+}$ belongs to $\mathcal M$ 
 if and only if $m$ satisfies the symmetry $J m^{-1} J = m$, and the geodesic $G(t)$ in $\mathcal L = \syms{2n+1}{+}$ satisfies the symmetry and thus it is a geodesic of $\mathcal M$.
 Since $\mathcal N$ can be realized in $\sym{n+1}{+}$ as in \eqref{eq:N},  
 the projection from $\mathcal M$ to the sub-matrix
 by the first $(n+1)$-rows and the first  $(n+1)$-columns can be understood as the Riemannian submersion $\pi:\mathcal M \to \mathcal N$.
 Finally, the geodesic $G(t)$ in $\mathcal M$ is in fact horizontal, 
 that is, a tangent vector of $G(t)$ at any point $p\in \mathcal M$ is
  orthogonal to the fiber $\pi^{-1}(\pi(p))$, and a standard argument about Riemannian submersions shows the following geometric characterization:
\begin{themain}\label{thm:main1} 
 Any geodesic in the multivariate normal manifold $(\mathcal N, g)$
 can be obtained by the Riemannian submersion of a horizontal geodesic of the 
 Riemannian symmetric space  $\mathcal M$ in \eqref{eq:RiemSym}.
\end{themain}
The proof of Theorem \ref{thm:main1} will be given in Section \ref{subsc:Rimsub} below.
\begin{Remark}
 A part of the result in Theorem \ref{thm:main1} had been
 obtained by a direct computation in \cite[Chapter 4]{Ino:theis}. In this paper we 
 give a geometric proof
 using the fundamental result of Riemannian submersions \cite{On, Har}.
\end{Remark}
 For the above characterization, 
 it is important that two Riemannian symmetric spaces $\mathcal L=\syms{2n+1}{+}= \SL/\SO{2n+1}{}$ and 
 $\mathcal M= \SO{n+1, }{n}/\operatorname{S}(\OO{n+1}{} \times \OO{n}{})$ are 
 given in the space of symmetric positive definite matrices with determinant $1$ 
$\syms{2n+1}{+}$.
 Moreover since $\mathcal M$ is embedded totally geodesically in $\mathcal L$, the 
 horizontal geodesic $G(t)$ can be thought as a geodesic in $\syms{2n+1}{+}$. 
 On the one hand, Nakamura \cite{Nak:AHM} had given an algorithm 
 to obtain the mid point of a geodesic segment connecting two points $p$ and $q$ in 
 the space of positive definite matrices, the so-called \textit{AHM algorithm} (Arithmetic and harmonic mean algorithm).
 Thus applying the AHM algorithm to the horizontal geodesic of $\mathcal M \subset \mathcal L$ 
 and by using the projection $\pi : \mathcal M \to \mathcal N$, we obtain
 the following result:
  \begin{themain}\label{thm:main2}
   Let $p, q \in \mathcal N$ and $\gamma$ be distinct points and 
  the geodesic segment connecting $p$ and $q$, respectively. 
  Moreover, let $r$ be the midpoint on $\gamma$. 
  Then there exist two sequences $\{p_n \in \mathcal N\}_{n=0,1,2,\dots}$ and $\{q_n \in \mathcal N\}_{n=0,1,2,\dots}$ 
  with $p_0=p$ and $q_0=q$
  such that $\{p_n\}_{n=0,1,2,\dots}$ and $\{q_n\}_{n=0,1,2,\dots}$ 
 both converge to the mid point $r$ of $\gamma$.
  \end{themain}
  The proof of Theorem \ref{thm:main2} will be given in Section \ref{subsc:AHM} below.
 
 In the above construction, the block Cholesky decomposition 
 has a key role. 
 As in \eqref{eq:MD},  for any $G \in \syms{2n+1}{+}$ there exist unique matrices 
 $M$ and $D$ such that $G = M D M^T$ holds, where $M$ is a block 
 lower triangular matrix with identity block diagonal matrices and $D$ 
 is a block diagonal matrix, respectively.
  Setting $G_1 = MD$ and $G_2 = M^T$,
  we have 
 \begin{equation}\label{eq:Chol}
 G = G_1 G_2 \quad \mbox{for any $G \in \syms{2n+1}{+}$},
 \end{equation}
 where $G_1$ is the unique block lower triangular matrix and 
 $G_2$ is the unique block upper triangular matrix with identity block diagonal 
 matrices, respectively.

 Since we have a horizontal geodesic $G(t) = \exp (t V)$ and the block Cholesky dedomposition, 
 it is natural to expect some integrable system (Lax pair) associate to them, see \cite[Chapter 8]{Gue}.
 Then the main result of the paper is the following:
\begin{themain}\label{thm:main3}
 Let $G(t)=\exp (t V)$ be the horizontal geodesic in \eqref{eq:GtiniV} 
 of the symmetric space $\mathcal M$ through the identity and set $G(t)=G_1(t) G_2(t)$ to 
be the block Cholesky decomposition as in 
 \eqref{eq:Chol}.
 Then the adjoint orbit 
 \begin{equation}\label{eq:L}
 L(t)=\ad (G_1(t)^{-1})  V:\R   \to \so{n+1,}{n},
\end{equation} 
 defines the following system of ODEs (Lax equation)
 in the Lie algebra $\so{n+1,}{n}$$:$
 \begin{equation}\label{eq:Lax}
 \dot L = [L, M], \quad 
 L = 
 \begin{pmatrix}
 -Q & r & 0 \\
 a_0^T & 0 & -r^T \\
 0 & -a_0 & Q^T 
 \end{pmatrix},
 \quad 
 M =
 \begin{pmatrix}
 -Q & 0 & 0 \\
 a_0^T & 0 & 0 \\
 0 & -a_0 & Q^T 
 \end{pmatrix},
 \end{equation}
 where $r: \R \to \R^n, Q: \R \to \operatorname{M}(n \times n, \R)$,
 and $a_0\in \R^n$ is the constant vector defined in \eqref{eq:GtiniV}.
  More explicitly \eqref{eq:Lax} is equivalent to the following system of ODEs
 for $r$ and $Q$$:$
 \begin{equation}\label{eq:Laxexplicit}
 \dot Q = - r a_0^T, \quad  \dot r =  Q r,
 \end{equation}
 and it is further equivalent to 
 \begin{equation}\label{eq:Laxexplicit2}
 2 \dot Q = Q^2 - A_0^2 -  2 a_0a_0^T,  \quad\dot r =  Q r,
 \end{equation}
 with the initial condition $Q(0) = A_0$ and $r(0)= a_0$.
 
 Conversely, any solution of \eqref{eq:Lax} with
 the initial condition 
 $L(0) = V$ given in \eqref{eq:GtiniV} 
 can be obtained by \eqref{eq:L}.
 \end{themain}
\begin{Remark}
 For $n=1$, the Lax equation becomes the finite non-periodic Toda lattice \cite[Chapters 5, 8]{Gue},
 and thus \eqref{eq:Lax} gives a Toda lattice 
 type Lax pair. The construction of the Toda lattice like Lax pair is similar to the Kostant-Toda lattices construction
 given in \cite{Kos}. However, the author would not know a precise relation.
\end{Remark}
 The proof of Theorem \ref{thm:main3} will be 
 given in Section \ref{sc:Lax} below.
\section{Geodesics and AHM algorithm}
\subsection{An explicit construction of geodesics}\label{subsc:explictgeod}
It is known \cite[Theorem 6.1]{Sko} that the Riemannian geodesic equations for $\mathcal N = (\Sigma, \mu)$ 
can be formulated as
\begin{align}\label{eq:Geo0}
\ddot \Sigma + \dot \mu \dot \mu^T - \dot \Sigma \Sigma^{-1} \dot \Sigma = 0, \quad 
\ddot \mu + \dot \Sigma \Sigma^{-1} \dot \mu = 0,
\end{align}
 where dot denotes derivative with respect to 
  a parameter $t$. 
 From the second equation, $\frac{d}{dt}(\Sigma^{-1} \dot \mu)=0$
 holds, that is, there exists some constant $a_0 \in \R^n$ such that 
 $\Sigma^{-1} \dot \mu =  a_0$
 holds. Moreover from the first equation, 
 $\frac{d}{dt}(\Sigma^{-1} \dot \Sigma + a_0 \mu^T )=0$
  holds, that is,  there exists some constant $A_0 \in \sym{n}{}$ 
  such that  $\Sigma^{-1} \dot \Sigma + a_0 \mu^T= A_0$
 holds.
 Thus the system \eqref{eq:Geo0} can be integrated 
 by an  initial point $(\id_p, 0)$ and an initial direction 
$(A_0, a_0) \in \sym{n}{} \times \R^n$.
 Introducing 
\begin{equation*}
\Theta = \Sigma^{-1}  \in \sym{n}{+} \quad \mbox{and}\quad  \delta = \Sigma^{-1} \mu \in \R^n, 
\end{equation*}
the geodesic equations 
in \eqref{eq:Geo0} can be formulated as 
\begin{equation}\label{eq:Geo2}
\frac{d}{dt} \begin{pmatrix}
 \Theta &  \delta
\end{pmatrix}
= \begin{pmatrix}
-A_0 & a_0
\end{pmatrix}
\begin{pmatrix}
\Theta &  \delta \\
 \delta^T & 1 + \|\delta \|^2_{\Theta^{-1}}
\end{pmatrix},
\quad \Theta(0)= \id, \; \delta (0)=0, 
\end{equation}
where $\|\delta \|^2_{\Theta^{-1}} = \delta^T \Theta^{-1} \delta$.
From the form of \eqref{eq:Geo2} it is natural to expect that it could be solved by 
exponential of $(A_0, a_0)$. 
Unfortunately, the simplest matrix exponential 
$\exp \left\{t \left(\begin{smallmatrix}- A_0 & a_0 \\
a_0 & 0 \end{smallmatrix}\right)\right\}$ is 
not a solution of \eqref{eq:Geo2} by 
the term $1 + \|\delta \|^2_{\Theta^{-1}}$
of the right hand side of \eqref{eq:Geo2}.
To understand the term $1+\|\delta \|^2_{\Theta^{-1}}$, let us consider the block Cholesky decomposition of 
the matrix in the right hand side of \eqref{eq:Geo2}:
\[
 \begin{pmatrix}
\Theta &  \delta \\
 \delta^T & 1 + \|\delta \|^2_{\Theta^{-1}}
\end{pmatrix} 
 = \begin{pmatrix}
  \id & 0 \\ \delta^T \Theta^{-1}&  1
 \end{pmatrix}
 \begin{pmatrix}
 \Theta & 0 \\
 0 & 1
 \end{pmatrix}
 \begin{pmatrix}
 \id & \Theta^{-1} \delta \\
  0 & 1
 \end{pmatrix}.
\]
 Then the second term in the right-hand side has the special property, that is, 
 the $(n+1, n+1)$-entry is just $1$. 
 Thus we would like to find a exponential 
 matrix which has such a middle term property in a larger matrix:
  Let 
 \begin{equation}\label{eq:J}
J=
\begin{pmatrix}
&&\id_{n}\\
& \id_{1}& \\
\id_{n}&&
\end{pmatrix},
\end{equation} 
 where  $\id_j$ denotes the identity matrix of degree $j$.
 In fact the following lemma gives an answer in order 
$2n+1$ positive definite matrices with determinant $1$, $\syms{2n+1}{+}$:
\begin{Lemma}[Lemma 7 in \cite{Ino:IJMI}]\label{lem:geo}
 Let $G\in \syms{2n+1}{+}$ with symmetry $JG^{-1}J = G$, where
 $J$ is in  \eqref{eq:J} and consider 
 the block Cholesky decomposition of 
 \[
 G = M D M^T, \quad M \in \mathcal L, D \in \mathcal A
 \] 
 where  
\begin{align*}
\mathcal L 
&= \left\{
\begin{pmatrix}
\id_n & 0 & 0\\
L_{21} & 1 & 0\\
L_{31} & L_{32} & \id_n
\end{pmatrix} \in \operatorname{GL}(2n+1, \R)
\right\}, \\
 \mathcal A
&= \left\{\begin{pmatrix}
A_1 & 0 & 0\\
0 & A_2 & 0\\
0 & 0 &A_3 
\end{pmatrix}\,\Big|\,
A_1, A_3 \in \sym{n}{+}, A_2 \in \sym{1}{+}
\right\}.
\end{align*}
 Then it
  can be computed as 
 \begin{equation}\label{eq:LandD}
 M = \begin{pmatrix}
  \id_n & 0  &0 \\ \tilde \delta^T \tilde \Theta^{-1}&  1 & 0 \\
   * &  -  \tilde \Theta^{-1} \tilde \delta & \id_n \\
 \end{pmatrix}, \quad 
  D=
 \begin{pmatrix}
 \tilde \Theta & 0  & 0 \\
 0 & 1 & 0  \\
 0 & 0 & \tilde \Theta^{-1} 
 \end{pmatrix},
  \end{equation}
  where $\tilde \Theta  \in \sym{n}{+}$ and $\tilde \delta \in \R^n$
  are matrix valued functions of $t$, and $*$ denotes some $n\times n$ matrix.
  \end{Lemma}
 \begin{proof}
  The symmetry $J G^{-1} J = G$ implies the forms of $L$ and $D$ in \eqref{eq:LandD}.
 \end{proof}
 \begin{Remark}
  A proof of existence of the block Cholesky decomposition can be found in \cite[Lemma 5]{Ino:IJMI}. 
 \end{Remark}
 Since the matrix $V$  in \eqref{eq:GtiniV} has a symmetry $- J V J = V$, 
 $G(t) = \exp (t V)$ in \eqref{eq:GtiniV} has the symmetry $JG^{-1}J  = G$.
 Therefore the block lower triangular matrix $M$ and the block diagonal matrix 
 $D$ have the forms as in \eqref{eq:LandD}.
 Let $H$ be a sub-matrix of $G$ given by first $n+1$
 rows and first $n+1$ columns. From the forms of $M$ and $D$ in \eqref{eq:LandD}, it can be computed as
 \begin{equation}\label{eq:H}
   H = \pi(G) = \begin{pmatrix}
  \tilde \Theta & \tilde \delta \\
  \tilde \delta^T & 1 + \|\tilde \delta \|_{\tilde \Theta^{-1}}^2 
  \end{pmatrix},
 \end{equation}
 and it is easy to see that $H$ is 
 a Riemannian geodesic in $\mathcal N$.
\begin{Theorem}[\cite{Eri}]
 The sub-matrix $H(t)$ of $G(t)$ 
 defines a Riemannian geodesic through the initial point $(\id_p,0)$
  and the direction $(A_0, a_0)$.
\end{Theorem}
\begin{proof}
 The matrix valued function $G$ satisfies the
 following system of ODEs:
\[
\dot G = V G, \quad G(0) = \id.
\]
where $V$ is given in \eqref{eq:GtiniV}.
Thus the sub-matrix $H$ in \eqref{eq:H} satisfies the same system 
of the geodesic equation \eqref{eq:Geo2} for $(\Theta, \delta)$ and the uniqueness of ODEs implies that $\Theta = \tilde \Theta $ and $\delta = \tilde \delta$.
\end{proof}
\begin{Remark}
 The embedding of $\mathcal N$ into 
 $\sym{n+1}{+}$ as in \eqref{eq:N} had been considered in \cite{CO2, LMR}.
 Note that in there $(\mu, \Sigma) \in (\R^n, \sym{n}{+})$ had been embedded as
 \[
 \begin{pmatrix}
 \Sigma + \mu \mu^T & -\mu \\    
  -\mu^T & 1
 \end{pmatrix} \subset \sym{n+1}{+}.
 \]
 Our embedding is the inverse of it, that is,
 \[
 \begin{pmatrix}
 \Sigma + \mu \mu^T & -\mu \\    
  -\mu^T & 1
 \end{pmatrix} ^{-1}
 = \begin{pmatrix}
\Theta &  \delta \\
 \delta^T & 1 + \|\delta \|^2_{\Theta^{-1}}
\end{pmatrix}. 
 \]
 Geodesics of the multivariate normal manifold 
 had been first computed by \cite{Yo}
 for $n=1$. In fact for $n=1$, the multivariate normal manifold becomes 
 the two-dimensional hyperbolic space $\mathbb H^2$.
\end{Remark}
\subsection{Riemannian submersion}\label{subsc:Rimsub} 
 We would like to  understand geometrically
 the construction of the previous section. 
 First it is clear that $\SL$ acts on $\syms{2n+1}{+}$ by 
\begin{equation}\label{eq:actSL}
  g\mapsto g P g^T, \quad g \in \SL, \quad P \in \syms{2n+1}{+}.
\end{equation}  
 Then the stabilizer at identity is given by 
\begin{align*}
  \operatorname{Fix}_{\tau}(\SL) &=
\left\{ g \in \SL \mid \tau (g) = g\right\}  = \SO{2n+1}{},
 \end{align*}
  where the involution $\tau$ is given by 
 \begin{equation}\label{eq:tau}
  \tau(g) = (g^{T})^{-1}, \quad g \in \SL.
 \end{equation}
 Thus $\syms{2n+1}{+}$ can be realized as a non-compact 
 Riemannian symmetric space of type \textrm{AI}:
\[
\mathcal L = \SL/\SO{2n+1}{}.
\]
 By using the involution in $J g^{-1}J$ for $\syms{2n+1}{+}$
 as in Lemma \ref{lem:geo} and the action in \eqref{eq:actSL}
 one can induce an involution 
 of $\SL$ as follows:
\begin{equation}\label{eq:sigmagroup}
    \sigma (g) = J g^{-T} J, \quad g \in \SL,
\end{equation}
 and the fixed point set
 of the involution $\sigma$ can be computed as
  \begin{align*}
  \operatorname{Fix}_{\sigma} (\SL) &=
\left\{ g \in \SL \mid \sigma (g) = g\right\}  = \SO{n+1,}{n},
 \end{align*}
 see \eqref{eq:invlie} for the Lie algebra computation.
 Since $\sigma$ in \eqref{eq:sigmagroup} and $\tau$ in \eqref{eq:tau}
 commute, one can consider 
 a fixed point set of $\tau$ on $\operatorname{Fix}_{\sigma}(\SL) = \SO{n+1,}{n}$, which is $\operatorname{S}(\OO{n+1}{} \times \OO{n}{})$, and the 
 corresponding homogeneous space is given by 
   \begin{equation*}
\mathcal M= \SO{n+1,}{n}/\operatorname{S}(\OO{n+1}{} \times \OO{n}{}) \subset 
\mathcal L = \syms{2n+1}{+},
\end{equation*}
 which is an another non-compact Riemannian symmetric space type {\rm BDI}.
 Note that a standard reference for symmetric spaces is \cite{Hel}. 
 
 On the Lie algebra level,  $\mathfrak g = \so{n+1,}{n}$ can be realized  in $\sl$ as follows:
Let $\sigma$ be an involution on $\sl$ given by 
\begin{equation}\label{eq:invlie}
\sigma (X) = - J X^T J, \quad X \in \sl
, \quad 
\end{equation}
 where $J$ is defined in \eqref{eq:J}. Note that $\sigma$ is a 
 derivative of the involution for $\SL$ in \eqref{eq:J} and 
$J^2 =\id_{2n+1}$ holds. The eigenvalues of 
$\sigma$ are $1$ with multiplicity $n+1$ and $-1$ with multiplicity $n$.
Thus the fixed point subalgebra $\mathfrak g = \operatorname{Fix}_{\sigma}(\sl)$ 
of $\sigma$ is isomorphic to $\so{n+1,}{n}$:
\begin{align}\label{eq:g}
 \mathfrak g &= 
 \left\{
 \begin{pmatrix}
 -Q & r & R \\
  t^T & 0 & -r^T \\
 S & -t & Q^T
 \end{pmatrix}\,\mid\,
Q\in \operatorname{M}(n\times n, \R), \,
R, S\in \skw{n}, \,
r, t \in \R^n
 \right\}, \\
  & \cong \so{n+1,}{n} \nonumber
\end{align}
 Moreover, let $\tau$ be an involution on $\mathfrak g= \so{n+1,}{n}$:
 \begin{equation*}
     \tau (X) = - X^T, \quad X \in \so{n+1, }{n}.
 \end{equation*}
 Then the fixed point subalgebra $\mathfrak k = \operatorname{Fix}_{\tau} (\mathfrak g)$
 can be computed as
 \begin{align}\label{eq:k}
 \mathfrak k &= 
 \left\{
 \begin{pmatrix}
 -Q & r & R \\
  -r^T & 0 & -r^T \\
 R & r & Q^T
 \end{pmatrix}\,\mid\,
Q, R\in \skw{n}, \,
r \in \R^n
 \right\}, \\
  & \cong \mathfrak{o}(n+1) \times \mathfrak{o}(n).
  \nonumber
\end{align}
The compliment subspace $\mathfrak m$ can be computed 
as 
\begin{equation}\label{eq:m}
\mathfrak m = 
 \left\{
 \begin{pmatrix}
 -Q & r & R \\
  r^T & 0 & -r^T \\
- R & -r & Q
 \end{pmatrix}\,\mid\,
Q\in \sym{n}{}, \,
R\in \skw{n}, \,
r \in \R^n
 \right\}.
\end{equation}
 Note that $\mathfrak m \subset \sym{2n+1}{}$.
 Thus by \eqref{eq:g}, \eqref{eq:k} and \eqref{eq:m}, we have 
\[
\mathfrak g = \mathfrak k \oplus \mathfrak m,
\]
 and $[\mathfrak k, \mathfrak k]\subset \mathfrak k$, $[\mathfrak m, \mathfrak m]\subset \mathfrak k$ and
 $[\mathfrak m, \mathfrak k] \subset \mathfrak m$
 hold. 
 It is evident that the tangent space of $\mathcal M$ at the identity 
 can be represented by $\mathfrak m$ in \eqref{eq:m}, see 
 \cite{Hel} in details.

 The Riemannian metric of $\mathcal M$ at identity can be induced from a Killing  
metric on $\mathcal L = \syms{2n+1}{+} = \SL/\SO{2n+1}{}$:
\[
 \langle A, B\rangle = \tr (A B) \quad \mbox{for $A, B \in \sym{2n+1}{}$.}
\]
 Moreover, any geodesic in $\mathcal M$ through identity is given by 
 $\tilde G(t) = \exp (t \tilde V)$ for some $\tilde V \in \mathfrak m$,
 and thus $G(t)$ in \eqref{eq:GtiniV} is a geodesic in $\mathcal M$.
 From Lemma \ref{lem:geo}, we know that an element $m\in \mathcal M\subset \sym{2n+1}{+}$ can be 
 represented by 
\begin{equation}\label{eq:elementM}
 m = 
  \begin{pmatrix}
  \Theta & \delta & G_{13}\\
  \delta^T & 1 + \|\delta \|_{\Theta^{-1}}^2  &G_{23} \\
  G_{13}^T & G_{23}^T & G_{33}
  \end{pmatrix}\in \mathcal M \subset \syms{2n+1}{+},
\end{equation}
 where $\Theta \in \sym{n}{+}, \delta \in \R^n, G_{13} \in {\rm M}(n \times n, \R), 
 G_{23}  \in {\rm M}(n \times 1, \R), G_{33} \in \sym{n}{}$. 
 Since the multivariate normal manifold  $\mathcal N = \left\{(\Sigma, \mu)\mid  \Sigma \in \sym{n}{+},\; \mu \in \R^n\right\}$ can be identified by \eqref{eq:N},
 thus $\pi : \mathcal M \to \mathcal N$ is naturally defined for $m \in \mathcal M$
 in \eqref{eq:elementM} by 
\[
    \pi (m) = 
\begin{pmatrix}
  \Theta & \delta \\
  \delta^T & 1 + \|\delta \|_{\Theta^{-1}}^2 
\end{pmatrix} \in \mathcal N,
 \]
 and it is a Riemannian submersion:  
 $d \pi|_{\id}: T_{\id} \mathcal M \to 
 T_{\id} \mathcal N$ with
 \[
  T_{\id} \mathcal M \ni \begin{pmatrix}
 -Q & r & R \\
  r^T & 0 & -r^T \\
 -R & -r & Q
 \end{pmatrix} \mapsto 
 \begin{pmatrix}
 -Q & r  \\
  r^T & 0
 \end{pmatrix}
 \in T_{\id} \mathcal N.
 \]
 Note that the Riemannian metric of $\mathcal N$ at identity can be computed as
 \[
 \langle A, B\rangle = 2 \tr (A B) \quad \mbox{for $A, B \in  T_{\id} \mathcal N \subset \sym{n+1}{}$,}
 \]
 see \cite{LMR, GQ}.
 For a Riemannian submersion $\pi: \mathcal M \to \mathcal N$, a
 tangent vector to $\mathcal N$ at $p$ is \textit{horizontal} if it is orthogonal to 
 the fiber $\pi^{-1} (\pi(p))$.
 Then the horizontal subspace $\mathfrak h \subset \mathfrak m$ can be identified with 
\begin{equation*}
\mathfrak h = 
\left\{
  X = \begin{pmatrix}
 -Q & r & 0 \\
  r^T & 0 & -r^T \\
0 & -r & Q
 \end{pmatrix}
 \in \mathfrak m
\right\},
\end{equation*}
 and thus the geodesic $G(t)$  in $\mathcal M$ given in \eqref{eq:GtiniV} is in fact horizontal at identity, that is, $V \in \mathfrak h$.
 Then the following fact about Riemannian submersions is fundamental: 
\begin{Theorem}[Corollary 2 in \cite{On}]\label{thm:horizontal}
 Let  $\pi: \mathcal M \to \mathcal N$ be a Riemannian submersion.
 If a geodesic $\gamma$ of $\mathcal M$ that is horizontal at some one point, $\gamma$ is always horizontal 
 $($hence $\pi \circ \gamma$ is a geodesics of $\mathcal N)$.
\end{Theorem}
\begin{Remark}
 The result first was proved by Hermann \cite{Har} by using length-minimizing properties and O'Neill \cite{On} gave a simpler proof.
\end{Remark}
 The geodesic $\gamma$ in Theorem \ref{thm:horizontal} is called the \textit{horizontal 
 geodesic}, and thus the curve $G(t)$ in \eqref{eq:GtiniV} 
 is the horizontal geodesic in $\mathcal M$. 
\begin{proof}[Proof of \textrm{Theorem \ref{thm:main1}}]
 From Theorem \ref{thm:horizontal}, 
 we understand the curve $G(t)$ is a horizontal geodesic in $\mathcal M$
 and any geodesic in $\mathcal N$ can be obtained by the Riemannian submersion 
 of a horizontal geodesic $G(t)$. This completes the proof.
\end{proof}
 
\subsection{Arithmetic harmonic mean algorithm}\label{subsc:AHM}
 It is known that $\mathcal M$ is a totally geodesic submanifold in $\mathcal L = 
 \syms{2n+1}{+}$, see \cite{Le}.  Since geodesics in $\mathcal N$
 are obtained by the projection of horizontal geodesics in $\mathcal M$, 
 therefore one can think $G(t)$ as a geodesic in $\mathcal L = \syms{2n+1}{+}$ 
 and use it to compute geodesics  in $\mathcal N$.
 In \cite[Theorem 10]{Nak:AHM}, Nakamura gave an algorithm 
 to obtain the midpoint $R_0$ of the Riemannian geodesic from 
 $P_0$ and $Q_0$ in the space of positive definite matrices.
 First note that the midpoint $R_0$ of $P_0$ and $Q_0$ is 
 given by 
 \[
 R_0 = P_0^{1/2}\left(P_0^{-1/2} Q_0 P_0^{-1/2}\right)^{1/2} P_0^{1/2},
 \]
  and in particular if $P_0 = \id_{2n+1}$, then $R_0 = Q_0^{1/2}$ holds.
 For $n=0, 1, 2, \dots$, define
 \begin{align}\label{eq:AHMalg}
  P_{n+1} &= \frac12 (P_n + Q_n),\\   
  Q_{n+1} &= 2 (P_n^{-1} + Q_n^{-1})^{-1}.
\label{eq:AHMalg2}
 \end{align}
 Then the sequences of matrices $\{P_n\}_{n=0, 1, 2, \dots}$ and $\{Q_n\}_{n=0, 1, 2, \dots}$ defined by \eqref{eq:AHMalg} and \eqref{eq:AHMalg2} 
 are all positive definite.
 \begin{Theorem}[Theorem 10 in \cite{Nak:AHM}]\label{thm:Nakamura}
     The sequences $\{P_n\}_{n=0, 1, 2, \dots}$ and $\{Q_n\}_{n=0, 1, 2, \dots}$ 
     tend to the midpoint 
     $R_0 = P_0^{1/2}(P_0^{-1/2} Q_0 P_0^{-1/2})^{1/2} P_0^{1/2}$ of 
      the geodesic segment connecting $P_0$ and $Q_0$ in a quadratic order.
 \end{Theorem}
\begin{Remark}
 Since  $\{P_n\}_{n=0, 1, 2, \dots}$ and $\{Q_n\}_{n=0, 1, 2, \dots}$
 in Theorem \ref{thm:Nakamura} satisfy 
\begin{align*}
 Q_{n+1}-P_{n+1} &= \frac{1}{2} (Q_n - P_n)^2 (Q_n + P_n)^{-1}, 
\end{align*}
 see \cite[p.171]{Nak:AHM}, thus they converge to $R_0$ in 
 a ``quadratic order''.
\end{Remark}
 We now prepare the following Lemma.
 \begin{Lemma}\label{lem:sym}
 Let $Q_0, P_0 \in \syms{2n+1}{+}$. Then the sequences of matrices 
 $\{P_n\}_{n=0, 1, 2, \dots}$ and $\{Q_n\}_{n=0, 1, 2, \dots}$ given by 
 \eqref{eq:AHMalg} and \eqref{eq:AHMalg2} also take values in $\syms{2n+1}{+}$.
 \end{Lemma}
\begin{proof}
 By the constructions in \eqref{eq:AHMalg} and \eqref{eq:AHMalg2}, 
 it is clear that $Q_n$ and $P_n$ take values in $\sym{2n+1}{+}$.
 We show that $Q_n$ and $P_n$ have determinant $1$.

 For the case of $Q_0 = \id$: By Lemma 7 in \cite{Nak:AHM}, $Q_n P_n = P_n Q_n$ for 
 $n=0, 1,2, \dots$ and therefore $Q_n$ and $P_n$ can be simultaneously diagonalized 
 and if they have determinant $1$ then $Q_{n+1}$ and $P_{n+1}$ also have 
 determinant $1$.

 For the general case of $Q_0$: 
 By (30) in   \cite{Nak:AHM}, the the sequences of matrices 
 $\{P_n'\}_{n=0, 1, 2, \dots}$ and $\{Q_n'\}_{n=0, 1, 2, \dots}$ are introduced by 
\[
 P_n'=Q_0^{-1/2} P_n Q_0^{1/2}, \quad
 Q_n'=Q_0^{-1/2} Q_n Q_0^{1/2}, \quad n =0, 1, 2, \dots,
\]
 and clearly $Q_0'=\id$. Thus $Q_{n}'$ and $P_{n}'$ are determinant $1$ as before,
 and thus $Q_{n}$ and $P_{n}$ are determinant $1$ as well.
\end{proof}
 The algorithm defined in \eqref{eq:AHMalg} and \eqref{eq:AHMalg2} together with Lemma \ref{lem:sym} will give an algorithm for the 
 mid point of geodesic segment in $\mathcal N$:
  \begin{proof}[The proof of Theorem \ref{thm:main2}]
  Let $p, q \in \mathcal N$ be distinct points and the geodesic segment $\gamma$ connecting $p$ and $q$.
  Let $P_0$ and $Q_0$ be distinct points and the horizontal geodesic $G(t)$ connecting $P_0$ and 
   $Q_0$ in $\mathcal M$ such that $\pi(P_0) =p, \pi(Q_0) =q$ and $\pi (G(t)) = \gamma$.  Since $\mathcal M \subset \mathcal L = 
 \syms{2n+1}{+}$ is totally geodesic, the $P_0, Q_0$ and $G$ can be points and a geodesic in 
 $\syms{2n+1}{+}$. Now applying Theorem \ref{thm:Nakamura}, one obtains 
 sequences of matrices $\{P_n\}_{n=0, 1, 2, \dots}$ and $\{Q_n\}_{n=0, 1, 2, \dots}$
 such that they converge to the midpoint 
\[
 R_0=  P_0^{1/2}(P_0^{-1/2} Q_0 P_0^{-1/2})^{1/2} P_0^{1/2}
\]
 of $G(t)$ in a quadratic order. Moreover by Lemma \ref{lem:sym},
 $\{P_n\}_{n=0, 1, 2, \dots}$ and $\{Q_n\}_{n=0, 1, 2, \dots}$ take values in $\syms{2n+1}{+}$.
 Then the Riemannian submersion 
 $\pi: \mathcal M \to \mathcal N$ gives a corresponding
 sequences of points $\{p_n=\pi(P_n) \in \mathcal N\}_{n=0, 1, 2, \dots}$ and $\{q_n=\pi(Q_n) \in \mathcal N\}_{n=0, 1, 2, \dots}$ such that they 
 converge to the midpoint  $r = \pi(R_0)$ of the geodesic segment $\gamma\subset \mathcal N$ connecting $p, q \in \mathcal N$.
  \end{proof}
 
\section{Toda lattice type Lax pair}\label{sc:Lax}
 Let us rephrase the block Cholesky decomposition of $G \in  \SO{n+1,}{n} 
  \subset \syms{2n+1}{+}$  in Lemma \ref{lem:geo} as follows:
 \begin{equation}\label{eq:Cholesky2}
 G  = G_1 G_2, \quad G_1 =M D, \quad G_2= M^T, 
 \end{equation}
 where $M$ is the block lower triangular matrix with trivial diagonal
 and $D$ is the block diagonal matrix.
 More explicitly,
 \begin{equation*}
 M=
 \begin{pmatrix}
  \id_n & 0  &0 \\ m_{21}&  1 & 0 \\
   m_{31} &  - m_{21}^T& \id_n \\
 \end{pmatrix}, \quad 
  D=
 \begin{pmatrix}
 d_{11} & 0  & 0 \\
 0 & 1 & 0  \\
 0 & 0 & d_{11}^{-1}
 \end{pmatrix},
  \end{equation*}
 where $m_{21}$ is a row vector with length $n$, 
 $d_{11}$ is a order $n$ positive definite symmetric matrix and
  $m_{31}$ is a $n \times n$ matrix.
 Note that from the symmetry $JG^{-1}J = G$, $m_{31}$ and $m_{21}$ satisfy a symmetry relation
 \[
 m_{31} + m_{31}^T = - m_{21} m_{21}^T.
 \]
  Let us denote the projection to the first factor of the above decomposition 
 by $\pi_1$. Accordingly, the Lie algebra $\mathfrak g= 
 \so{n+1,}{n}$ of $\SO{n+1,}{n}$ can be 
 decomposed as
 \[
 \mathfrak g = \mathfrak g_1 \oplus \mathfrak g_2,
 \]
  that is,
 \begin{equation*}
  \pi_1 X = 
  \begin{pmatrix}
  -Q & 0 & 0 \\
  t^T & 0 & 0 \\
  S & t & Q^T 
    \end{pmatrix}
    \quad
    \mbox{for}
    \quad
    X = 
    \begin{pmatrix}
  -Q & r & R \\
  t^T & 0 & -r^T \\
  S & -t & Q^T 
    \end{pmatrix} 
    \in \mathfrak g= \so{n+1,}{n}.
 \end{equation*}
 Let $V$ be the matrix in \eqref{eq:GtiniV} and the horizontal geodesic 
$G(t) = \exp (t V)$, and consider the block Cholesky decomposition 
$G(t) = G_1(t) G_2(t)$  by \eqref{eq:Cholesky2}. Note that since $G(t)$ takes 
values in $\syms{2n+1}{+}$, the decomposition is defined for all $t \in \R$.
Then consider an adjoint orbit 
\begin{equation*}
 L(t)=\ad (G_1(t)^{-1})V: \R \to \mathfrak g = \so{n+1,}{n}.
\end{equation*}
 We are now ready to prove the main result.
  \begin{proof}[Proof of \textrm{Theorem \ref{thm:main3}}]
 We first show that $L(t)$ has the form in \eqref{eq:Lax}.
 Since $G_1(t)$ is a block lower triangular matrix, it is easy to see 
 that the upper right part, that is, the sub-matrix given by the first 
 $n$-rows and the last $n$-columns, is zero. On the other hand, $L(t)$
  can be rephrased as
\begin{equation*}
 L(t)=\ad (G_1(t)^{-1})V= \ad (G_2(t) G(t)^{-1})V= \ad (G_2(t))V,
\end{equation*}
 since $G(t)= \exp (t V)$ commutes with $V$.
 Obviously, $G_2(t)$ is the block upper triangular and has a 
 trivial block diagonal part, and therefore the lower 
 left part of $L$, that is, the sub-matrix given by the last 
 $n$-rows and the first $n$-columns, is zero.
 Moreover a straightforward computation shows that the 
 $(n+1, n+1)$ entry is zero and the sub-matrices 
 given by $n+1$-th row  and $n+1$-th column are
 constants $a_0^T$ and $-a_0$, respectively.
 
 Finally explicit systems of ODEs in \eqref{eq:Laxexplicit} and \eqref{eq:Laxexplicit2} are given by straightforward computations.
\end{proof}
The author states that there is no conflict of interest. Data sharing is not applicable to this article as no new data were created or analyzed in this study.

\subsection*{Acknowledgements}
We would like to express our gratitude for the anonymous referee's comments on the manuscript, particularly for their valuable input in the substantial reformulation of Theorem B.

\bibliographystyle{plain}
\def\cprime{$'$}

\end{document}